\documentclass{article}
\usepackage{graphicx} 
\usepackage{amsfonts}
\usepackage{amsthm}
\usepackage{enumitem} 
\usepackage{amsmath}
\usepackage{amssymb}
\usepackage{authblk}
\usepackage{hyperref}
\usepackage{xcolor}
\usepackage{todonotes}
\usepackage{soul}
\usepackage[all]{xy}
\usepackage{subcaption}

\title{New divergence measures between persistence diagrams and stability of vectorizations}
\author[1]{Alessandro Bravetti}
\author[2]{Martín Mijangos\thanks{$^\ast$Corresponding author. Email: martinmij@gmail.com}}
\author[3]{Pablo Padilla}
\affil[1]{School of Science and Technology, University of Camerino, Via Madonna delle Carceri, 9, Camerino, 62032, Italy}
\affil[2,3]{Instituto de Investigaciones en Matem\'aticas Aplicadas y en Sistemas, Universidad Nacional Aut\'onoma de M\'exico, Ciudad de M\'exico, M\'exico}
\date{}

\newcommand{\D}{\mathcal{D}}

\newcommand{\N}{\mathbb{N}}
\newcommand{\f}{\longrightarrow}

\newcommand{\R}{\mathbb{R}}

\newtheorem{theorem}{Theorem}[section]

\newtheorem{lemma}[theorem]{Lemma}

\newtheorem{definition}[theorem]{Definition}

\newtheorem{remark}[theorem]{Remark}
\newtheorem{example}[theorem]{Example}

\begin{document}

\maketitle

\begin{abstract}
Given a filtration of simplicial complexes, one usually applies persistent homology and summarizes the results in barcodes. 
Then, in order to extract statistical information from these barcodes, one needs to compute statistical indicators over the bars of the barcode.
An issue with this approach is that usually infinite bars must be deleted or cut to finite ones; 
however, so far there is no consensus on how to perform this procedure. 
In this work we propose for the first time a systematic way to analyze barcodes through the use of statistical indicators. 
Our approach is based on the minimization of a divergence measure that generalizes the standard Wasserstein or bottleneck distance to a new asymmetric distance-like function that we introduce and which is interesting on its own. 
In particular, we analyze the topology induced by this divergence and the stability of known vectorizations  with respect to
this topology.
\end{abstract}

\section{Introduction}
 Topological data analysis is a recently-developed area of mathematics~\cite{edelsbrunner2016topological, otter2017roadmap} where tools such as persistent homology are used to analyze data. The result of applying persistent homology is summarized in the so-called \emph{persistent diagrams} and \emph{persistence barcodes}. Commonly these structures are so complex that it is impossible to extract information just by a simple glance. Therefore one needs to apply some statistics or to map them to vector spaces where statistics or machine learning models can be applied. Common statistics over the persistence barcodes are computed over the lengths of the bars. Typical examples are the mean, the variance or the Shannon entropy of the distribution of the lengths of the bars~\cite{mijangos2024persistent,rucco2014entropy}.  
 A problem with this approach is that it only directly applies to barcodes where bars have finite lengths. To overcome this issue, the infinite bars are usually dropped or cut to have finite lengths, without a general consensus on how to perform this procedure. 
 
 In this work we propose a way to reduce a ``full'' persistence diagram 
 (a diagram including all the bars, either of finite or of infinite length) 
 to a ``reduced'' one (one with only bars of finite lengths) which is the ``closest'' to the original one in a precise sense. To do so, we propose the minimization of a certain distance-like function that is motivated by information theory~\cite{cover1999elements,csiszar2022information,polyanskiy2014lecture} and information geometry~\cite{amari2016information,ay2017information,nielsen2020elementary}.
Indeed, even though there exists already distance functions 
defined over the set of persistence barcodes, such as the Wasserstein or the bottleneck distance \cite{computational_top}, 
these take infinite values whenever the number of infinite bars between the two barcodes is different, which is the typical situation. 
Therefore, motivated by the use in information theory and information geometry of asymmetric, distance-like functions called ``divergence measures'', 
we introduce a particularly well-suited generalization of the Wasserstein distance which is a divergence measure over the set of barcodes, or equivalently, over the set of persistence diagrams. The most important property of this new divergence measure is that it allows us to measure finite (asymmetric) distances between two barcodes, even when they have a different number of infinite bars. 

 The paper is organized as follows. In Section \ref{sec:intro} we recall the relevant concepts from topological data analysis. 
 In Section \ref{sec:divergences} we define a family of divergences which we call \textit{$(f,p)$-Wasserstein divergences}. 
 These will allow us 
 in Section~\ref{sec:stat} to look for an optimal way to cut the infinite bars in a barcode. 
  Finally, by leveraging these divergences, we will extend continuous functions defined over the set of persistence diagrams with finite coordinates to the set of all persistence diagrams, thus obtaining a kind of stability result that constitutes the main finding of this work. 

\section{Background}\label{sec:intro}
In this section we recall some concepts of topological data analysis. More details can be found in \cite{edels-harer, computational_top, otter2017roadmap}.

\subsection{Persistent homology}
When dealing with data analysis we typically want a way to assess the importance of different aspects.  
For instance, we may want to identify noise from real features, 
or we may want to pay special attention to some features within a certain threshold. 
This is what persistent homology is useful for: while
homology alone captures the topological features of a ``static'' simplicial complex, 
the key idea of persistent homology is to take simplicial complexes at different scales to  
represent our data, and then assess the importance of different features based on 
how long they persist when we move 
along all scales.
We make this idea precise as follows.
\begin{definition}\rm
A \textit{filtration} of  a simplicial complex $K$ is a sequence of nested simplicial 
complexes $K_0\subseteq K_1\subseteq\cdots\subseteq K_n\subseteq K$.
Moreover, we set $K_i=\emptyset$ for $i<0$ and $K_i=K$ for $i>n$.
\end{definition}

An inclusion $\rho^{i,j}:K_i\hookrightarrow K_j$, $i\leq j$, induces a map 
$$(\rho^{i,j}_*)_p:H_p(K_i)\longrightarrow H_p(K_j)$$ 
for all $p\in\mathbb{Z}$, where $H_p(K_i)$ is the $p$th homology group of $K_i$.
To simplify the notation, from now on
we will drop the subscript~$*$ and just denote these maps as $\rho^{i,j}_p$. 
Given a class $c$ in $H_p(K_i)$ (which can be considered as a $p$-dimensional hole) 
we can track its persistence as we move along  the filtration by means of the maps $\rho^{i,j}_p$
according to the next definition.

\begin{definition}\rm
Given a homology class $c\in H_p(K_i)$ we define its birth and death as~\cite{edels-harer}.
\begin{itemize}
\item[i)] We say that $c$ \textit{is born} at $K_i$ if $c\in H_p(K_i)\backslash\{0\}$ and $c\notin \mbox{im}(\rho^{i-1, i}_p).$
\item[ii)] We say that $c$ \textit{dies entering}  $K_j$ 
if $\rho^{i, j-1}_p(c)\notin \mbox{im}(\rho^{i-1,j-1}_p)$ but $\rho^{i, j}_p(c)\in \mbox{im}(\rho^{i-1, j}_p).$
\end{itemize}
Then \emph{the persistence of $c$} is the difference $d-b$ where $b$ represents the level 
in the filtration where $c$ was born and $d$ the value of the parameter where it died. 
If a class never dies we take $d=\infty$.
The multiset given by the intervals $[b, d)$ of all the $p$-homology classes appearing in a filtration $F$ is called 
\textit{the p-persitence barcode}, denoted by $Bc_p(F)$. 
\end{definition}
An alternative representation of the barcodes are the persistence diagrams. 
The \textit{$p$-persistence diagram} associated to $F$, denoted $Dgm_p(F)$, 
is the multiset in $\bar{\R}^2$ consisting of all the points $(b, d)$ where $[b, d)\in Bc_p(F)$,
together with
all the points of the form $(x, x)\in \R^2$, which are considered with infinite multiplicity. 
More generally, we can define a persistence diagram as follows.
\begin{definition}\rm
    A persistence diagram is the union of a finite multiset of points $(x, y)\in \bar{\R}^2$ where $x<y$,
    with the set of points of the form $(x, x)\in \bar{\R}^2$, taken with infinite multiplicity.
\end{definition}

\subsection{Metrics over the space of persistence diagrams}
Denote by $\D$ the set of all persistence diagrams. 
There are two well-known metrics defined over this set, the Wasserstein and bottleneck distance, 
defined as follows~\cite{computational_top}. 

\begin{definition}\rm
    Let $1\leq p<\infty$ and $A, B\in \D$. The \emph{$p$-Wasserstein distance} $d_p$ is defined as
    \begin{equation}\label{eq:wasserstein}
    d_p(A, B)=\left(\inf_\gamma\sum_{a\in A}||a-\gamma(a)||_\infty^p \right)^{1/p}
    \end{equation}
    where the infimum runs over all the bijections $\gamma:A\f B$ and
    we define as usual $$ ||a-b||_\infty=\max\{|x_1-x_2|, |y_1-y_2|\}$$
    for $a=(x_1, y_1)$ and $b=(x_2, y_2)$,
    setting $|y_1-y_2|$ to zero if both $y_1$ and $y_2$ are infinite.
\end{definition}   
\begin{definition}\rm
    Given $A, B\in \D$, the \emph{bottleneck distance} $d_\infty$ is defined as 
    $$d_\infty(A, B)=\inf_{\gamma}\sup_{a\in A}\{||a-\gamma(a)||_\infty\}.  $$
\end{definition}





\section{Divergences over the space of perstistence diagrams}\label{sec:divergences}\rm
In this section we introduce a new family of divergence measures defined over the space of persistence diagrams.

Consider two persistence diagrams $A$ and $B$. It is straightforward to see that $d_p(A, B)$ is finite for any $p\geq 1$ if and only if
$A$ and $B$ have the same number of points with infinite death times, which we will call \textit{infinite points} from now on.
We now define a divergence between persistence diagrams whose value can be finite even when the diagrams have a different number of infinite points. 


\begin{definition}\label{flass}\rm
Let $f:[0, \infty]\longrightarrow [0,\infty)$ be a continuous function which satisfies the properties
\begin{flalign}\label{eq:conditions}
\begin{split}
&\mbox{{\rm i)} }f(x) \leq x \mbox{ for all } x\in [0, \infty], \\
 &\mbox{{\rm ii)} }f(x+y)\leq f(x)+f(y) \mbox{ for any } x,\,y\geq 0\,.
\end{split}&&
\end{flalign}
Then we say that $f$ is \emph{a sub-diagonal and sub-additive function}.
\end{definition}

\begin{definition}\rm
    Let $\lambda:\bar\R^2\longrightarrow \bar\R^2$ be the function defined by 
$$\lambda(x, y):=\left(\frac{x+y}{2}, \frac{x+y}{2}\right),$$ that is, $\lambda$ is the projection onto the diagonal. 
\end{definition}

\begin{definition}\label{def:divergence}\rm
    Let $1\leq p<\infty$ and $A, B\in \D$. We define the
    \emph{$(f,p)$-Wasserstein divergence} from $A$ to $B$, denoted $D_p^{f}(A||B)$, as
    \begin{equation}\label{eq:def_div}
    D_p^{f}(A||B)=\left(\inf_\gamma\left\{\sum_{a\in A}
    ||a-\gamma(a)||_\infty^p+\sum_{b\in B\setminus \gamma(A)}f(||b-\lambda(b)||_\infty)^p\right\}\right)^{1/p}     
    \end{equation}
where the infimum runs over all the injections $\gamma$ from $A$ to $B$,
$\lambda$ is the projection to the diagonal and 
$f$ is a sub-diagonal and sub-additive function.
\end{definition}


Note that the $(f,p)$-Wasserstein divergence is not necessarily symmetric. In order to emphasize this, 
the arguments of the function $D_p^{f}$ are delimited by vertical bars instead of a comma.  
We can define the \textit{$(f,p)$-bottleneck divergence}, $D_\infty^{f}$, generalizing in an analogous way the bottleneck distance.

\begin{example}\label{ex:Ex1}\rm
    Let 
    $$f=\frac{e^x-1}{e^x+1}$$ 
    be the translation of the logistic function and
    consider the persistence diagrams $A$ and $B$ of Figure \ref{fig:div}. 
    Note that the only points outside the diagonal in these diagrams are $(2, 10)$ and $(5, \infty)$ in $A$ and the point $(3, 11)$ in $B$. 
    \begin{figure}[h!]
        \centering
        \includegraphics[scale=0.43]{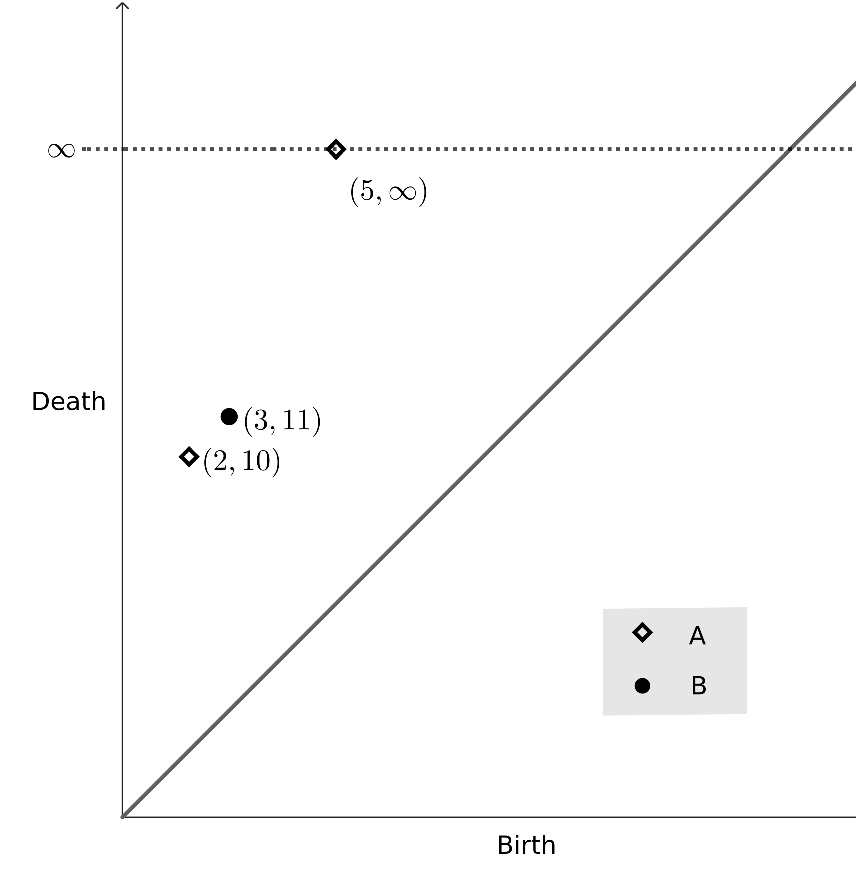}
        \caption{Persistence diagrams in Example~\ref{ex:Ex1}.}
        \label{fig:div}
    \end{figure}
    Moreover, every injection from $A$ to $B$ maps the infinite point $(5, \infty)$ to a finite point in $B$.
    Then the first sum in~\eqref{eq:def_div} is always infinite, which in turn implies that 
    $$D_p^{f}(A||B)=\infty$$ for all $p\geq 1$. 
    On the other hand, it is quite easy to see that the minimizing injection 
    from $B$ to $A$ is the one that fixes the diagonal and maps the point $(3, 11)$ to $(2, 10)$. In this case have that 
    \begin{align*}
        D_p^{f}(B||A)&=\left(||(2, 10)-(3,11)||_\infty^p+f(||(5, \infty)-(\infty, \infty)||_\infty)^p\right)^{1/p}\\
                &=(1+f(\infty)^p)^{1/p}\\
                &=\sqrt[p]{2}.
    \end{align*}
    Therefore we conclude that the $(f,p)$-Wasserstein divergence is in general not symmetric and, 
    more importantly, it can take finite values even when the 
    Wasserstein distance $d_p(A, B)$ is infinite for all $p\geq 1$.  
\end{example}

In the next theorem we summarize some useful properties of this family of divergences.
\begin{theorem}\label{thm:properties_div}
    Let $p\geq 1$ and $A, B, C\in \D$. $D_p^{f}(-|| -)$ has the following properties:
    \begin{enumerate}
        \item[\rm{(i)}] $D_p^{f}(A || B)\geq 0$.
        \item[\rm{(ii)}] If $f(x)>0$ for all $x>0$, $D_p^{f}(A||B)=0$ if and only if 
        $A=B$.
        \item[\rm{(iii)}] $D_p^{f}(A || B)\leq d_p(A, B)$ and $D_p^{f}(B || A)\leq d_p(A, B)$.
        \item[\rm{(iv)}] $D_p^{f}(A||C)\leq D_p^{f}(A||B)+D_p^{f}(B||C)$ for any $B\in \D$.
    \end{enumerate}
\end{theorem}
Properties (i) and (ii) imply that $D_p^f$ is indeed a divergence function, according
to the standard definition in information theory~\cite{cover1999elements}. Property (iii) means that $D_p^f$ 
is always bounded from above by the corresponding $p$-Wasserstein distance. Property (iv) is the triangle 
inequality for $D_p^f$.
\begin{proof}
    \begin{enumerate}[label=(\roman*)]
        \item This is clear from the definition.
        \item If $A=B$, the identity function 
        realizes the infimum on the right hand side of \eqref{eq:def_div}, and thus $D_p^{f}(A||B)=0$.
        If $D_p^{f}(A||B)=0$ and $\gamma:A\f B$ is an injection that realizes $D_p^{f}(A||B)$, 
        by the hypothesis on $f$ we have that all the elements in the second summand in~\eqref{eq:def_div} have to be zero.
        Therefore, $\gamma$ has to be a bijection. Then, since $||\cdot||_\infty$ is a norm, the only way to have $D_p^{f}(A||B)=0$ is that 
        $\gamma$ is the identity function, implying that $A=B$.
        \item This follows from the fact that the set of bijections between $A$ and $B$ is a subset of the injections from $A$ to $B$ and 
        of those from $B$ to $A$.
        \item Let $\gamma:A\f B$ and $\delta: B\f C$ be injections that realize $D_p^{f}(A||B)$ and $D_p^{f}(B||C)$ respectively. 
        Intuitively, the trick is to add extra points in order to change injections for bijections and thus write the divergence in terms of only one sum with the same value. 
         More precisely, for an element $c\in C\setminus\delta(B)$, 
        let $b_c'\in \R^2$ be a point such that $||b_c'-c||_\infty=f(||c-\lambda(c)||_\infty)$ and denote by $B'$ the multiset $\{b'_c|c\in C\setminus\delta(B)\}$. Analogously, for $b\in (B\setminus\gamma(A))\cup B'$, let $a'_b\in \R^2$ be such that $||a'_b-b||_\infty=f(||b-\lambda (b)||_\infty)$ and denote by $A'$ the multiset $\{a'_b|b\in (B\setminus\gamma(A))\cup B'\}$. 
        Now let $\gamma':A\cup A'\f B\cup B'$ and $\delta':B\cup B'\f C$ be extensions of $\gamma$ and $\delta$ respectively such that $\gamma'(a'_b)=b$ and $\delta'(b'_c)=c$ (See Figure \ref{fig:tr_ineq}). 
        \begin{figure}[h!]
            \centering
            \begin{subfigure}[b]{0.44\textwidth}
            \centering
            \includegraphics[width=\textwidth]{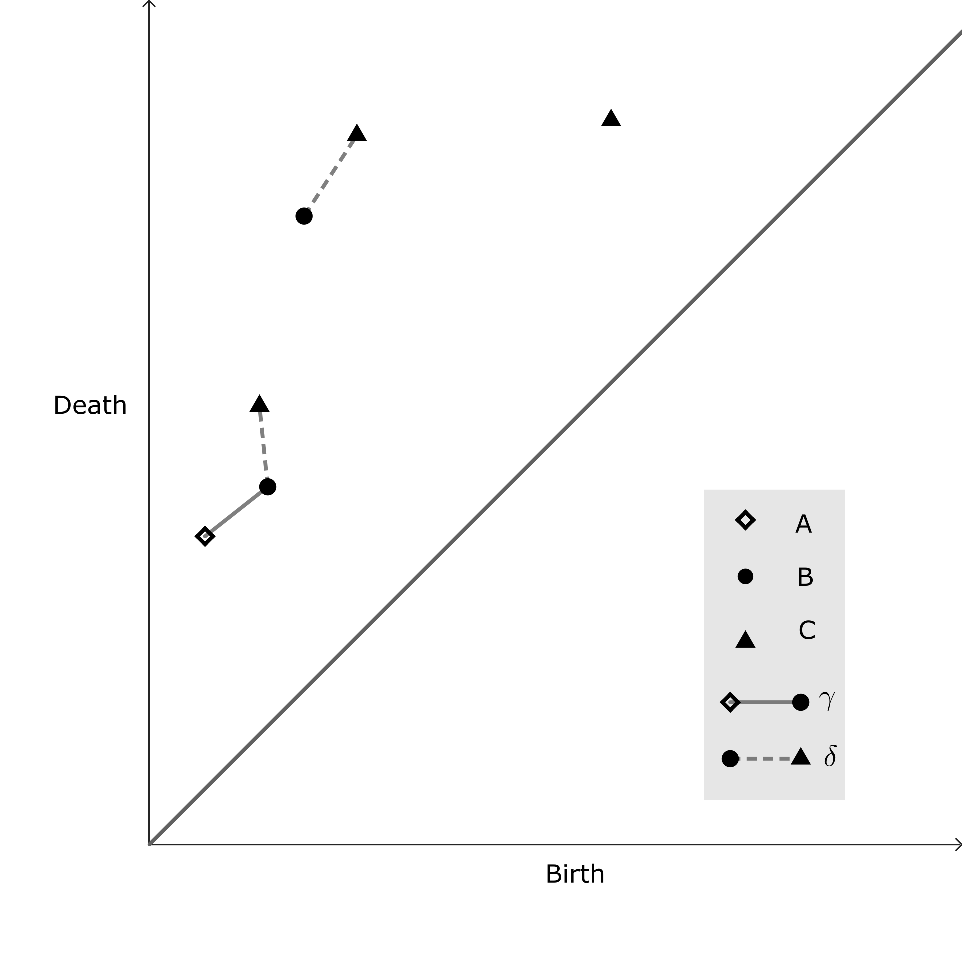}
            \end{subfigure}
             \hfill
            \begin{subfigure}[b]{0.44\textwidth}
            \centering
            \includegraphics[width=\textwidth]{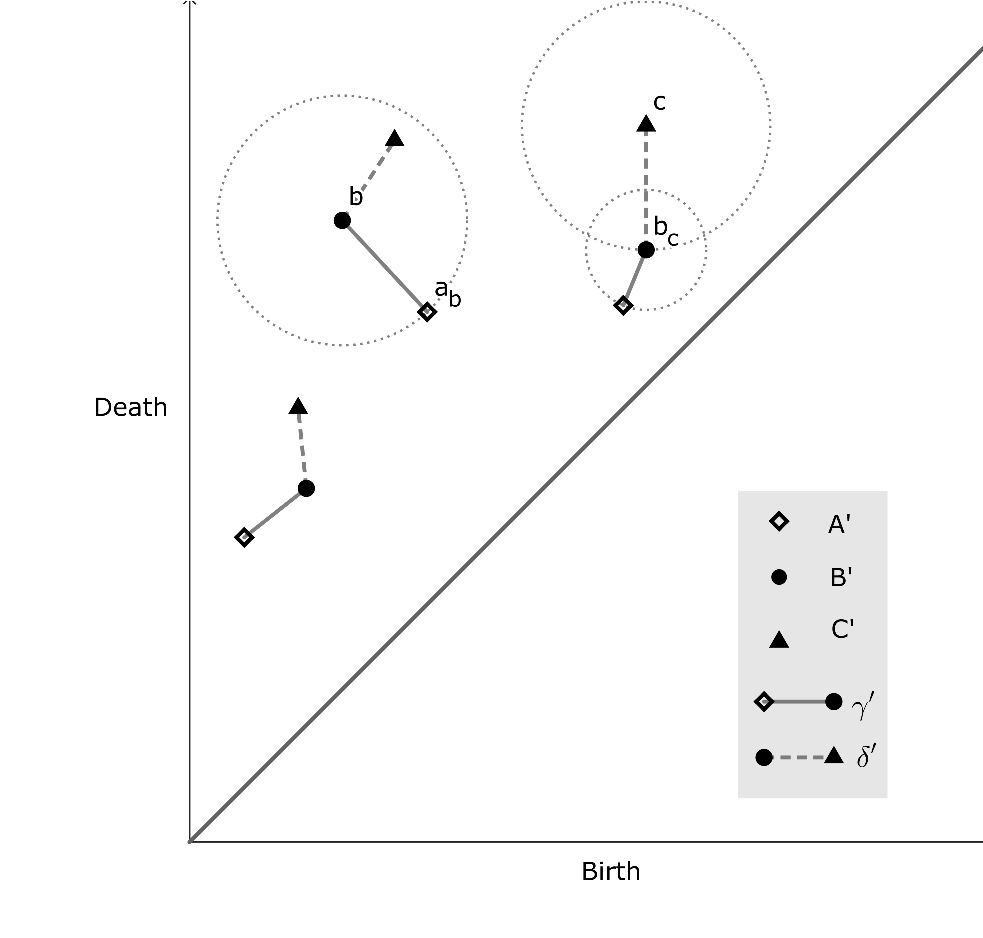}
            \end{subfigure}
           \caption{In the left panel we have the persistence diagrams $A$, $B$, and $C$. 
           The realizing injection of the divergence is represented by joining the points. In the right panel extra
           points are added in order to transform the injection into a bijection as described in the text.  }
           \label{fig:tr_ineq}
        \end{figure}
        
        It follows that 
        $$D_p^{f}(A||B)=\left(\sum_{a\in A\cup A'}||a-\gamma'(a)||_\infty^p\right)^{1/p}$$
        and
        $$D_p^{f}(B||C)=\left(\sum_{b\in B\cup B'}||b-\delta'(b)||_\infty^p\right)^{1/p}.$$ 
        Since $\gamma'$ is a bijection (of multisets), we can rewrite the index in the previous equality as
        $$D_p^{f}(B||C)=\left(\sum_{a\in A\cup A'}||\gamma'(a)-\delta'(\gamma'(a))||_\infty^p\right)^{1/p}.$$
        Then, by the Minkowski inequality, we have
        \begin{multline}\label{eq:sum_divergences}
         D_p^{f}(A||B)+D_p^{f}(B||C)\geq \\
         \left(\sum_{a\in A\cup A'}(||a-\gamma'(a)||_\infty+||\gamma'(a)-\delta'(\gamma'(a))||_\infty)^p\right)^{1/p}.
        \end{multline}
        Since $A\cup A'$ is a disjoint union, we can split the sum on the right in this inequality 
        into a sum indexed by $A$ and a sum indexed by $A'$. Regarding the sum indexed by $A$, by the triangle inequality we have
        \begin{equation}\label{eq:sum_a}
        \sum_{a\in A}(||a-\gamma'(a)||_\infty+||\gamma'(a)-\delta'(\gamma'(a))||_\infty)^p\geq\\
        \sum_{a\in A}||a-\delta'(\gamma'(a))||_\infty^p.
        \end{equation}
        Regarding the sum indexed by $A'$, let us consider $a'\in A'$ and  $b=\gamma'(a')$ and
        take into account two cases.\\
        \noindent \textit{Case 1: $b\in B\setminus \gamma(A)$}. 
        In this case, by definition of the elements in $A'$, we have
        $$||a'-\gamma'(a')||_\infty+||\gamma'(a')-\delta'(\gamma'(a'))||_\infty=f(||b-\lambda (b)||_\infty)+||b-\delta'(b)||_\infty.$$
        By the properties i) and ii) of $f$ in Definition~\ref{flass},
        $$f(||b-\lambda (b)||_\infty)+||b-\delta'(b)||_\infty\geq f(||\delta'(b)-\lambda (b)||_\infty\geq f(||\delta'(b)-\lambda (\delta'(b))||_\infty).$$
        \noindent \textit{Case 2: $b\in B'$}. 
        In this case, again by the definition of the elements in $A'$ and $B'$, we have
        \begin{multline*}
            ||a'-\gamma'(a')||_\infty+||\gamma'(a')-\delta'(\gamma'(a'))||_\infty=
                        f(||b-\lambda (b)||_\infty)+\\f(||\delta'(b)-\lambda(\delta'(b))||_\infty)\geq f(||\delta'(b)-\lambda(\delta'(b))||_\infty).
        \end{multline*}
        By the two previous cases, it follows that
        \begin{multline}\label{eq:sum_a'}    
        \sum_{a'\in A'}(||a'-\gamma'(a')||_\infty+||\gamma'(a')-\delta'(\gamma'(a'))||_\infty)^p\geq\\
        \sum_{a'\in A'}f(||\delta'(\gamma'(a'))-\lambda(\delta'(\gamma'(a')))||_\infty)^p
        \end{multline}
        Substituting \eqref{eq:sum_a} and \eqref{eq:sum_a'} into \eqref{eq:sum_divergences} and noting that 
        $C\setminus \delta\circ\gamma(A)=\delta'(\gamma'(A'))$, we obtain
        \begin{multline*}
            D_p^{f}(A||B)+D_p^{f}(B||C)\geq\\ \left(\sum_{ A}||a-\delta'(\gamma'(a))||_\infty^p+\sum_{C\setminus \delta\circ\gamma(A)}f(||c-\lambda(c)||_\infty)^p\right)^{1/p}.
        \end{multline*}
        Since $\delta\circ\gamma$ is an injection from $A$ into $C$ and $D_p^{f}(A||C)$ minimizes 
        the sum on the right in the above inequality over all such injections, the result follows.

    \end{enumerate}
\end{proof}

In the next theorem we show that a particular ``symmetrization'' of the $(f,p)$-Wasserstein divergences 
is bounded both from above and from below by (multiples of) the corresponding $p$-Wasserstein distance. 
In the case of the 
bottleneck distance, this implies that $d_p$ can be completely recovered from knowledge of both $p$-divergences.

\begin{theorem}
    Given any $A, B\in \D$ and $1\leq p<\infty$, 
    \begin{equation}\label{eq:bounds}
        d_p(A, B)\leq (D_p^{f}(A || B)^p+ D_p^{f}(B||A)^p)^{1/p}\leq 2^{1/p}d_p(A,B).
    \end{equation}
    In particular, in the limit as $p$ tends to infinity, 
    $$\max\{D_\infty(A || B), D_\infty(B||A)\}=d_\infty(A, B).$$ 
\end{theorem}
\begin{proof}
Take $1\leq p<\infty$ and consider injections $\gamma:A\f B$ and $\eta:B\f A$ that realize $D_p^{f}(A||B)$ and $D_p^{f}(B||A)$ respectively. 
From the Cantor–Schröder–Bernstein Theorem, these induce a bijection $\nu: A\f B$, which can be constructed as~\cite{csb_teo}, 
$$\nu(a)=
\begin{cases}
    \gamma(a),\quad\quad \mbox{if $a\in C$}\\
    \eta^{-1}(a), \quad \mbox{if $a\in A\setminus C$}
    \end{cases}$$
where $C_1=A\setminus \eta(B)$,  $C_{k+1}=\eta(\gamma(C_k))$ and $C=\cup C_k$. On the one hand, as $d_p$ takes the minimum over all the bijections, $d_p(A, B)\leq (\sum_{a\in A}||a-\nu(a)||_\infty^p)^{1/p}$. Since  $\sum_{a\in A}||a-\nu(a)||_\infty^p\leq \sum_{a\in A}||a-\gamma(a)||_\infty^p+\sum_{b\in B}||b-\eta(b)||_\infty^p$ 
and $\gamma$ and $\eta$ realize $D_p^{f}(A||B)$ and $D_p^{f}(B||A)$ respectively, we have that
 $$d_p(A, B)\leq (D_p^{f}(A||B)^p+ D_p^{f}(B||A)^p)^{1/p}\,,$$
 thus proving the first inequality in~\eqref{eq:bounds}.
On the other hand, from Theorem \ref{thm:properties_div}, 
we have that $D_p^{f}(A||B)^p+D_p^{f}(B||A)^p\leq 2d_p(A, B)^p$, which  
is the second inequality in~\eqref{eq:bounds}. 

\end{proof}

\begin{remark}\rm
Since for the trivial case $A=B$ all the terms in the inequalities in~\eqref{eq:bounds}  vanish, the inequalities are tight for
any function $f$ which is sub-diagonal and sub-additive (Definition~\ref{flass}). 
Non-trivial cases where the equality holds are the following:
the right inequality becomes an equality for given $A$ and $B$ in $\D$ whenever
the minimizing injection happens to be a bijection. 
On the other hand, there exist persistence diagrams that turn the left inequality  
into an equality when the function $f$ is identically zero,
 as we show in the next example. 
\end{remark}

\begin{example}\label{ex:Ex3}\rm
    Consider the diagrams in Figure \ref{fig:pd}. 
       \begin{figure}[h!]
    \centering
    \includegraphics[scale=0.43]{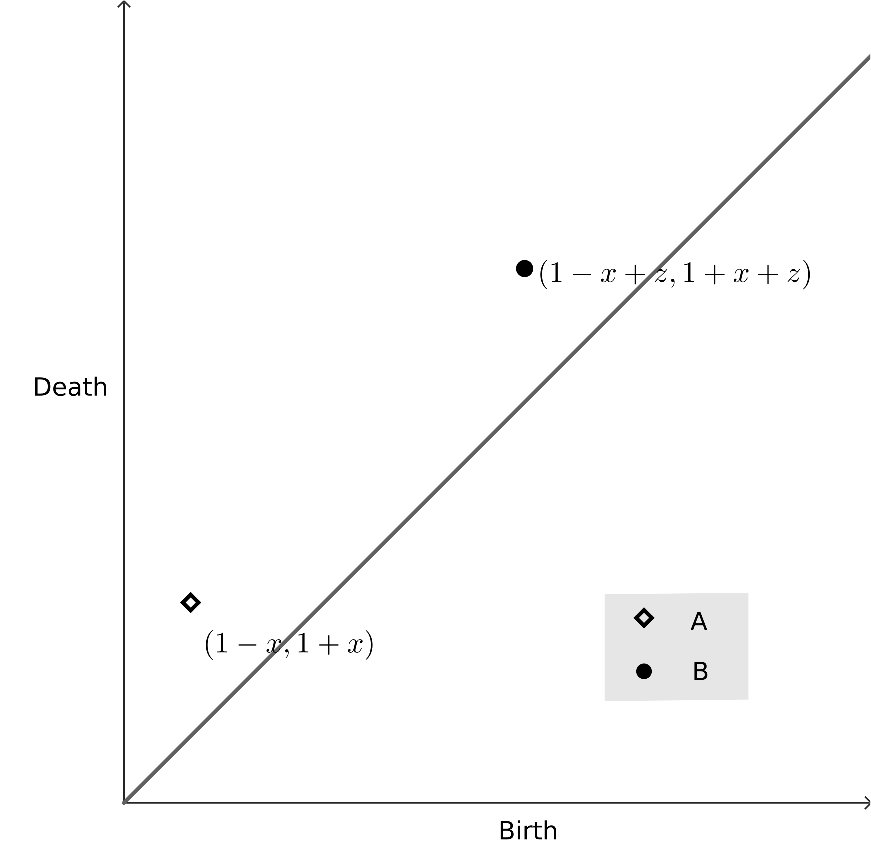}
    \caption{Persistence diagrams $A$ and $B$ of Example~\ref{ex:Ex3}.}
    \label{fig:pd}
\end{figure}
    $A$ is the persistence diagram with only one point not in the diagonal, $a=(1-x, 1+x)$ with $0<x<1$, 
    and $B$ is a translation of $A$ in the direction of $(z, z)$, $z>0$, that is, the only point of $B$ not in the diagonal 
    is $b=(1-x+z, 1+x+z)$.
We have two feasible injections $\gamma:A\hookrightarrow B$, $\gamma(a)=b$ or $\gamma(a)=(1, 1)$, 
this last one corresponding to the projection onto the diagonal. It can be seen that $||a-b||_\infty=z$ and $||a-(1,1)||_\infty=x$. 
If we consider $z> \sqrt[p]{2}x$, the second option is the minimizing injection and then $D_p^{f}(A||B)=x$. 
Analogously we can see that $D_p^{f}(B||A)=x$. 
To compute $d_p(A, B)$, we have two feasible bijections $\nu:A\f B$, the one induced by 
$\gamma$ and $\eta$, for which $(\sum_{a\in A}||a-\eta(a)||_\infty^p)^{1/p}=\sqrt[p]{2}x$, 
and the one where $\eta(a)=b$, for which $(\sum_{a\in A}||a-\eta(a)||_\infty^p)^{1/p}=z$. Since $z>\sqrt[p]{2}x$, 
the first bijection minimizes the sum and $d_p(A, B)=\sqrt[p]{2}x$. This shows that the first inequality is tight. 
\end{example}


In the next section we will prove that the $(f,p)$-Wasserstein divergence induces a topology over $\D$.
Then we will use this topology in order to prove a stability result for vectorizations defined directly on $\D$.

\section{Topology over $\D$ and stability of vectorizations}\label{sec:stat}
In order to apply statistics or machine learning methods over the set of persistence diagrams, $\D$ is usually mapped to normed spaces. 
Examples of these mappings are persistence images \cite{pers_images}, landscapes \cite{pers_landscapes}, entropy \cite{rucco2014entropy}, 
or simple statistics over the lengths of the bars, such as the mean or the standard deviation.
Commonly these definitions are well defined only if all the bars have finite length. 
Thus usually infinite bars are cut to obtain finite ones. However,
there is no systematic way to perform this procedure. 
We propose to cut the infinite bars so that the $(f,p)$-Wasserstein divergence is minimized
and we show in the next result that this procedure is well defined and coincides precisely with removing the bars of infinite length.

\begin{theorem}
    Let $\D_F$ be the set of persistence diagrams such that all points outside the diagonal have finite coordinates. Given $A\in \D$,
    there exists a unique $B\in \D_F$ such that $D_p^{f}(B|| A)=\inf\{D_p^{f}(C|| A)|C\in \D_F\}$. 
    Moreover, $B$ is the subset of $A$ obtained by removing all the points with infinite coordinates. 
\end{theorem}
\begin{proof}
    Let $B\in \D_F$ and $\gamma:B\f A$ be an injection that realizes $D_p^{f}(B||A)$, that is
    \begin{equation}\label{eq:min_div}
       D_p^{f}(B||A)=\left(\sum_{b\in B}||b-\gamma(b)||_\infty^p+\sum_{a\in A\setminus \gamma(B)}f(||a-\lambda(a)||_\infty)^p\right)^{1/p}. 
    \end{equation}
    If the infinite points of $A$ are in the image of $\gamma$, 
    then the first sum in \eqref{eq:min_div} is infinite. Therefore, since $\gamma$ is minimizing, we can safely assume that this is not the case,
    and thus all the infinite points of $A$ contribute to the second sum, independently of the choice of $B\in\D_F$. Now the best choice of $B$ in order
    to minimize at the same time both the first and the second sum is to choose $B$ as the subset of $A$ given by taking only the finite points of $A$. 
    Indeed, in this way $\gamma$ can be taken as the identity and the first sum vanishes, while the second sum runs only over the infinite points of 
    $A$, which is the least we can do, by the previous argument.
\end{proof}

\begin{definition}\label{def:projection}\rm
    Given $A\in \D$, let $\bar A\in \D_F$ be the persistence diagram that minimizes $D_p^{f}(-||A)$. 
    We call $\bar{A}$ \emph{the projection of $A$ onto $\D_F$} and denote by $\pi_F:\D\f \D_F$ the function  that takes $A$ onto $\bar{A}$.
\end{definition}

The next lemma shows that $\pi_F$ is a contraction, both with respect
to $D_p^f$ and with respect to $d_p$.

\begin{lemma}\label{lemma:projection_is_contraction}
    Given $A, B\in \D$ and for any $1\leq p\leq \infty$, 
    $$D_p^{f}(\pi_F(A)|| \pi_F(B))\leq D_p^{f}(A||B)$$ and 
    $$d_p(\pi_F(A),  \pi_F(B))\leq d_p(A, B).$$
\end{lemma}
\begin{proof}
For the first inequality, take an injection that realizes $D_p^{f}(A||B)$. 
This induces an injection between $\pi_F(A)$ and $\pi_F(B)$, which implies the result. 

For the second inequality  we divide the proof into two cases.\\
\textit{Case 1: $A$ and $B$ have a different number of infinite points.} 
In this case, $d_p(A, B)$ is infinite while $d_p(\pi_F(A), \pi_F(B))$ is finite.\\
\textit{Case 2: $A$ and $B$ have the same number of infinite points.} 
Let $\gamma$ be a bijection between $A$ and $B$ that realizes $d_p(A, B)$. 
This induces a bijection $\gamma'$ between $\pi_F(A)$ and $\pi_F(B)$ for which
$$\sum_{a\in A}||a-\gamma(a)||_\infty\geq \sum_{a\in \pi_F(A)}||a-\gamma'(a)||_\infty$$
and
$$\sup_{a\in A}\{||a-\gamma(a)||_\infty\}\geq \sup_{a\in \pi_F}(A)\{||a-\gamma'(a)||_\infty\}$$
since $\pi_F(A)\subset A$.
These inequalities imply $d_p(A, B)\geq d_p(\pi_F(A), \pi_F(B))$ 
for the case $1\leq p<\infty$ and $p=\infty$ respectively.
\end{proof}

Now that we have a well-defined projection of $\D$ onto $\D_F$, 
we are able to extend any map defined on $\D_F$ to all $\D$.
More precisely,
\begin{definition}\rm
   Given any vectorization
   $$g:\D_F\f X$$ 
   with $X$ a normed space, 
we define \emph{the extension of $g$} as the map
$$\tilde{g}:\D\f X$$ 
given by 
$$\tilde{g}=g\circ\pi_F.$$ 
\end{definition}

An important issue to take into account at this step, is to guarantee the stability of the process. 
Usually, stability results come in the form of Lipschitz continuity of the vectorizations $g$, that is, 
showing the existence of a $k\in \R^+$ such that
$$d(g(A), g(B))\leq k\,d_p(A, B)$$
for $A, \, B\in \D_F$, and with $d$ being the metric in $X$ induced by the norm and $d_p$ any of the $p$-Wasserstein distances, $1\leq p \leq \infty$ (see for instance \cite{pers_images,stability_ent,pers_landscapes}). 

If we want $\tilde{g}$ to be a useful vectorization, we need to show the stability of this map. 
An issue in this direction is that $\D$ equipped with the divergence $D_p^f$ is not a metric space, 
and therefore the usual arguments cannot be applied.
Nevertheless, it is possible to endow $\D$ with a topology. 
\begin{lemma}
   Let $$\beta:=\{B(A, r)|A\in \D,\,\, r\in \R^+\}$$ 
   where $B(A, r)=\{C\in \D|D_p^{f}(A||C)< r\}$. 
   Given $U$, $V\in \beta$, for all $X\in U\cap V$, there exists $W\in \beta$ such that $X\in W\subseteq U\cap V$.
\end{lemma}
\begin{proof}
    This is a consequence of (iv) in Theorem \ref{thm:properties_div}, that is, of the triangle inequality property of $D_p^f$.
\end{proof}
Since $\beta$ covers $\D$, the previous lemma implies that $\beta$ generates a topology on $\D$.  
\begin{definition}\rm
    Let $\tau_*$ be the topology of $\D$ generated by the basis $\beta=\{B(A, r)|a\in \D,\,\, r\in \R^+\}$ where $B(A, r)=\{C\in \D|D_p^{f}(A||C)< r\}$.
    Considering that in $B(A, r)$, $A$ is in the first argument of $D_p^f$ and $C$ is in the second one, we call this topology the 
    \emph{topology of $\D$ generated by the first balls}, or \emph{first topology of $\D$} in short.
\end{definition}

Now we are ready to prove that $\tilde{g}$ is stable in a continuity sense 
with respect to $\tau_*$ whenever the function $f$ 
defining the divergence $D_p^f$ satisfies an additional condition, given in the following definition. 

\begin{definition}\rm
    We say that a function $f:[0, \infty]\f [0, \infty]$ is \emph{$p$-increasing} if for every $n\in \N$ and $x=(x_i)\in\R_+^n$, 
    \begin{equation}\label{eq:condition}
        f(||x||_p)\leq ||(f(x_i))||_p\,.
    \end{equation}
\end{definition}

\begin{theorem}[Stability result for extensions of vectorizations]\label{thm:stability}
    Let $f:[0, \infty]\f [0, \infty)$ be a sub-diagonal and sub-additive function which is strictly increasing  
    and $p$-increasing. 
    Let $g:(\D_F, \tau)\f (X, \tau_X)$ be a continuous map where $\tau$  is the topology induced by the Wasserstein distance $d_p$ for some integer $1\leq p\leq \infty$ and $\tau_X$ is the topology induced by the metric on $X$. 
    Furthermore, let 
    $$\pi=\pi_F\circ\pi',$$ 
    where $\pi':(\D,\tau_*)\f (\D, \tau)$ 
    is the identity at the set level, and $\pi_F:(\D, \tau)\f (\D_F, \tau)$ is the projection in Definition~\ref{def:projection}.
    Then the map $$\tilde{g}:(\D, \tau_*)\f (X, \tau_X)$$ given by $$\tilde{g}=g\circ\pi$$ is continuous. 
\end{theorem}
We summarize the previous theorem in the following diagram.
     $$
        \xymatrix
   { (\D\ar[d]_{\pi}, \tau_*) \ar@{-->}[dr]^{\tilde{g}} & \\
     (\D_F \ar[r]_{g}, \tau) &  (X,\tau_X) }
     $$

\begin{proof}
    Since $g$ is continuous, it is enough to prove that $\pi:(\D,\tau_*)\f (\D, \tau) $ is continuous. 
    We recall that $\pi=\pi_F\circ\pi'$. 
    Moreover, since $\pi_F$ is continuous by Lemma \ref{lemma:projection_is_contraction}, 
    we just need to prove that $\pi'$ is continuous. 
    It is easy to see that continuity of $\pi'$ can be put in epsilon-delta terms, 
    that is, $\pi'$ is continuous if and only if for every $A\in \D$ and 
    $\epsilon>0$ there exist $\delta >0$ such that if $D_p^{f}(A||B)<\delta$ then $d_p(A, B)<\epsilon$. 
    We are going to prove that $\pi'$ is continuous in this way. Let $A\in\D$ and $\epsilon$ be a positive real number. 
    Now, the conditions on $f$ guarantee that it is a homeomorphims onto its image. 
    In particular, its inverse is continuous at 
    zero, which in turn implies that there exists $\delta>0$ such that if $f(x)<\delta$, then $x<\epsilon$. 
    Now let $B\in\D$ be a persistence diagram such that $D_p^{f}(A||B)<\delta$ and let $\gamma:A\f B$ 
    the injection that realizes $D_p^{f}(A||B)$. Since $f(x)\leq x$, we have the following inequality 
    \begin{align}
        \delta>&\left(\sum_{a\in A}||a-\gamma(a)||_\infty^p+\sum_{b\in B\setminus \gamma(A)}f(||b-\lambda(b)||_\infty)^p\right)^{1/p}\\
        \geq& \left(\sum_{a\in A}f(||a-\gamma(a)||_\infty)^p+\sum_{b\in B\setminus \gamma(A)}f(||b-\lambda(b)||_\infty)^p\right)^{1/p}  
    \end{align}
\end{proof}
Note that the injection $\gamma$ induces a bijection $\gamma':A\f B$ whose inverse is the map $\eta:B\f A$ defined as 
\begin{equation*}
    \eta(b)=\left\{\begin{array}{ll}
         \gamma^{-1}(b)&\mbox{if } b\in \gamma(A)  \\
         \lambda(b)& \mbox{if } b\notin \gamma(A).
    \end{array} \right.
\end{equation*}
We then have that 
$$\left(\sum_{a\in A}f(||a-\gamma'(a)||_\infty)^p\right)^{1/p} <\delta.$$
Since only a finite number of terms in the sum is different from zero, by hypothesis we have that
$$f\left(\left(\sum_{a\in A}||a-\gamma'(a)||_\infty^p\right)^{1/p}\right)\leq\left(\sum_{a\in A}f(||a-\gamma'(a)||_\infty)^p\right)^{1/p}.$$
Because of the way in which we chose $\delta$, it follows that 
$$\left(\sum_{a\in A}||a-\gamma'(a)||_\infty^p\right)^{1/p}<\epsilon.$$ 
Finally, since $\gamma'$ is a bijection and the Wasserstein distance is defined as the minimum over all the bijections, $d_p(A, B)<\epsilon$, which completes the proof that $\pi'$ is continuous.

\begin{remark}
    When $p=1$ or $p=\infty$, the condition $f(||x||_p)\leq ||(f(x_i))||_p$ is satisfied for any 
    sub-diagonal and sub-additive function $f$.
    Indeed, since $||x||_1=x_1+\dots+x_n$, the case $p=1$ follows from the sub-additivity of $f$. 
    On the other hand, since $||x||_\infty=\max\{x_1,\dots,x_n\}$, the case $p=\infty$ follows from the monotonicity of $f$. 
\end{remark}

\begin{remark}
    Note that, since most of the stability theorems of vectorizations 
    are stated with respect to the bottleneck distance, 
    Theorem~\ref{thm:stability} guarantees the stability of the vectorization when we equip $\D$ with 
    the bottleneck divergence.
\end{remark}



\section{Conclusions and future work}
We have introduced the $(f,p)$-Wasserstein divergences, a family of divergence measures that generalize the famous $p$-Wasserstein distances 
to corresponding asymmetric distance-like functions.  
Using this, we have demonstrated for the first time that it is possible to extend the commonly-used statistical indicators for the examination
of the barcodes that result from the standard topological data analysis to indicators that are continuously defined even when one admits a different
number of bars of infinite lengths.  
Our results are interesting for two reasons: on the one hand, we have shown for the first time that it is possible to
show the stability directly with the original
barcodes (without the need of ``cutting'' them to reduce the bars of infinite length); on the other hand, the introduction of the $(f,p)$-Wasserstein
family of divergence functions may open the door to the development of new theoretical and algorithmic tools based on the minimization of
divergences between these barcodes.

 \section*{Acknowledgements}
M.~Mijangos would like to thank CONAHCyT for the financial support. 
A.~Bravetti acknowledges financial support by DGAPA-UNAM, program PAPIIT, Grant No.~IA-102823. 
P.~Padilla would like to thank DGAPA (PASPA) and Clare Hall at the University of Cambridge.

\bibliographystyle{plain}
\bibliography{references}
\end{document}